\documentclass{amsart}

\usepackage{pstricks}
\usepackage{pstcol}
\usepackage{epsfig}
\usepackage{amsmath}
\usepackage{amsfonts}

\newtheorem{theorem}{Theorem}
\newtheorem{conjecture}[theorem]{Conjecture}

\author{Philippe Duchon}
\thanks{Research supported by French ANR
project MARS (BLAN06-2-134516). 2000 \textit{Mathematical Subject
Classification:} Primary 05A15; Secondary 05B45, 82B20}

\title[Link patterns of quarter-turn symmetric FPLs]{On the link
pattern distribution of quarter-turn symmetric FPL configurations}
\address{Philippe Duchon - ENSEIRB - LaBRI, Universit\'e Bordeaux 1,
  351 cours de la Lib\'eration, F-33405 Talence}

\begin{document}
\begin{abstract}
We present new conjectures on the distribution of link patterns for
fully-packed loop (FPL) configurations that are invariant, or almost
invariant, under a quarter turn rotation, extending previous
conjectures of Razumov and Stroganov and of de Gier. We prove a
special case, showing that the link pattern that is conjectured to be
the rarest does have the prescribed probability. As a byproduct, we
get a formula for the enumeration of a new class of quasi-symmetry of
plane partitions.
\end{abstract}
\maketitle

\section{Introduction}

In this paper, we study configurations in the fully packed loop model,
or, equivalently, alternating-sign matrices, that are invariant or
almost invariant under a rotation of 90 degrees.  While the
enumeration of this symmetry class of alternating-sign matrices was
conjectured by Robbins~\cite{Rob00} and proved by
Kuperberg~\cite{Kup02} and Razumov and Stroganov~\cite{RazStr06},
their refined enumeration according to the link patterns of the
corresponding fully packed loop configurations seems to have avoided
notice so far. We conjecture very close connections between this
refined enumeration and the corresponding enumeration for half-turn
invariant configurations, as studied by de~Gier~\cite{deG05}. This is
yet another example of a ``Razumov-Stroganov-like'' conjecture,
suggesting a stronger combinatorial connection between 
fully-packed loop configurations and their link patterns than originally
conjectured in~\cite{RazStr04}.

The paper is organized as follows. In Section~\ref{sec:FPL_link}, we
recall a number of definitions and conjectures on FPLs and their link
patterns, and define a new class of ``quasi-quarter-turn-invarriant''
FPLs when the size is an even integer of the form $4n+2$. We formulate
a conjecture on the enumeration of these qQTFPLs. In
Section~\ref{sec:QTFPL_link}, we give new conjectures on the
distribution of link patterns of QTFPLs and qQTFPLs; these can be seen
as natural extensions of the previously known Razumov-Stroganov and
de~Gier conjectures on general and half-turn symmetric FPLs,
respectively. We prove special cases of our conjectures in
Section~\ref{sec:special_cases}; in the qQTFPL case this is achieved
by making an explicit connection with the enumeration of some new
class of plane partitions.

\section{Fully-packed loops and link patterns}
\label{sec:FPL_link}

\subsection{Fully-packed loop configurations}

A \emph{fully-packed loop configuration} (FPL for short) of size $N$
is a subgraph of the $N\times N$ square lattice\footnote{Here $N$
refers to the number of vertices on each side; vertices are given
matrix-like coordinates $(i,j)$ with $0\leq i,j\leq N-1$, the top left
vertex having coordinates $(0,0)$}, where each internal vertex has
degree exactly 2, forming a set of closed loops and paths ending at
the boundary vertices. The boundary conditions are the
\emph{alternating} conditions: boundary vertices
also have degree 2 when boundary edges (edges that connect the finite
square lattice to the rest of the $\mathbb{Z}^2$ lattice) are taken
into account, and these boundary edges, when going around the grid,
are alternatingly ``in'' and ``out'' of the FPL. For definiteness, we
use the convention that the top edge along the left border is always
``in''. Thus, exactly $2N$ boundary edges act as endpoints for paths,
and the FPL consists of $N$ noncrossing paths and an indeterminate
number of closed loops.

\begin{figure}[htbp]
  \begin{center}
    \epsfig{file=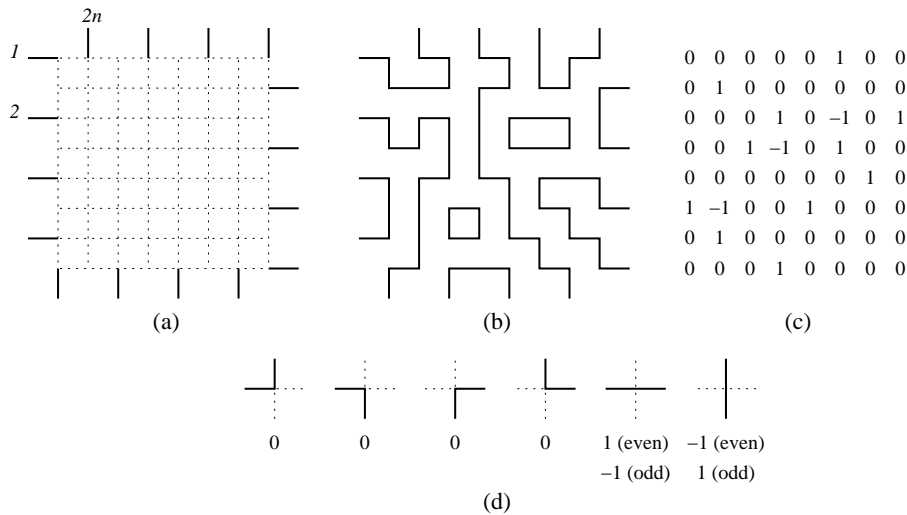,scale=0.8}
  \end{center}
\caption{(a) boundary conditions; (b) example FPL of
  size 8 and (c) corresponding ASM; (d) correspondence rules}
\label{fig:FPL_ASM_sample}
\end{figure}

FPLs are in bijection with several different families of discrete
objects, the most prominent in the mathematics literature being
\emph{alternating-sign matrices} of the same size. An alternating-sign
matrix has entries $0$, $1$ and $-1$, with the condition that, in each
line and column, nonzero entries alternate in sign, starting and
ending with a $1$ (changing the boundary conditions for the FPL would
correspond to changing the conditions on the first and last nonzero
entries for some or all lines and columns). This correspondence is
sketched in Figure~\ref{fig:FPL_ASM_sample}; the ``even'' and ``odd''
rules refer to the parity of the sum of line and colum indices. Other
objects include configurations of the \emph{square ice} or
\emph{6-vertex} model~\cite{Pro01}.

The enumeration formula for alternating-sign matrices of size $n$ was
proved in~\cite{Zei96,Kup96}:
\begin{equation}
  A(n) = \prod_{i=0}^{n-1} \frac{(3i+1)!}{(n+i)!};
\end{equation}
together with $A(1)=1$, this is equivalent to the recurrence
\begin{equation}
  \frac{A(n+1)}{A(n)} = \frac{n! (3n+1)!}{(2n)!(2n+1)}.
\end{equation}

The group of isometries of the square acts naturally on
alternating-sign matrices and on FPLs (with the caveat that some
isometries, depending on the parity of $N$, may exchange the ``in''
and ``out'' boundary edges, so that to have a given isometry act on
FPLs one may have to take the complement of the set of edges; for
rotations, this only happens when one performs a quarter-turn on FPLs
of odd size). As a result, for each subgroup of the full group of
isometries one may consider a \emph{symmetry class} of FPLs, which is
the set of FPLs that are invariant under the whole
subgroup. Enumeration formulae have been conjectured~\cite{Rob00} for
many classes, and some of them have been proved~\cite{Kup02}. In this
paper, we are only concerned with two classes: FPLs that are invariant
under a half-turn rotation (HTFPLs), and FPLs that are invariant under
a quarter-turn rotation (QTFPLs).

While HTFPLs of all sizes exist, QTFPLs are a slightly different
matter. QTFPLs of all odd sizes exist, but because for odd sizes the
90 degree rotation exchanges the boundary conditions, QTFPLs are
actually self-complementary (as edge sets) rather than invariant under
the rotation. QTFPLs of even size $N$ only exist if $N$ is a multiple
of $4$, which is easiest seen on the corresponding alternating-sign
matrices: the sum of entries in any quarter of the square has to be
exactly a quarter of the sum of all entries in the matrix, which is
equal to $N$.

For $N=4n+2$, while there are no QTFPLs of size $N$, we can define the
closest thing to it, which we call ``quasi-quarter-turn invariant
FPLs'' (qQTFPLs), and define as follows: an FPL of size $N=4n+2$ is a
qQTFPL if its symmetric difference with its image under a 90 degree
rotation is reduced to a single 4-cycle at the center of the grid;
furthermore, we require that a qQTFPL contain the two horizontal edges
of this center cycle. This last requirement is purely arbitrary:
accepting the alternative two vertical edges would simply double the
number of qQTFPLs, and not change the distribution of their link patterns
as we define them in Section~\ref{sec:QTFPL_link}.

The alternating-sign matrices corresponding to qQTFPLs are exactly
those which have quarter-turn invariance except for the four center
entries $a_{2n,2n}$, $a_{2n+1,2n}$, $a_{2n,2n+1}$, $a_{2n_1,2n+1}$,
which are bound only by the half-turn invariance rules. It is easy to
see that exactly two of these center entries will be $0$. This
particular class of ASMs does not seem to have been considered
previously in the literature, and their enumerating sequence does not
appear in the Online Encyclopedia of Integer Sequences~\cite{OEIS}.

Robbins~\cite{Rob00} conjectured, and Kuperberg~\cite{Kup02} proved,
among other things, that the numbers of FPLs of size $N$, HTFPLs of
size $2N$, and QTFPLs of size $4N$, are bound by the very intriguing formula
\begin{equation}
  A_{\textsc{QT}}(4N) = A_{\textsc{HT}}(2N) A(N)^2; 
\end{equation}
based on exhaustive enumeration up to $N=4$, we conjecture the
following similar formula:

\begin{conjecture}\label{conj:qQTFPL_comptage}
  The number of qQTFPLs of size $4N+2$ is
  \begin{displaymath}
    A_{\textsc{QT}}(4N+2) = A_{\textsc{HT}}(2N+1) A(N+1) A(N).
  \end{displaymath}
\end{conjecture}

Actually, a refined identity seems to hold, which nicely extends a
further conjecture of Robbins:
\begin{conjecture}\label{conj:qQTFPL_firstOne}
  Let $A(n;y)$ (respectively, $A_{\textsc{HT}}(n;y)$,
  $A_{\textsc{QT}}(n;y)$) denote the enumerating polynomial of FPLs
  (respectively, HTFPLs, qQTFPLs) of size $n$; each object is
  given weight $y^k$, where $k$ is the index of the column (numbered $0$
  to $n-1$) containing the single nonzero entry in the first line of
  the corresponding alternating-sign matrix; then for any $n\geq 1$,
  \begin{displaymath}
    A_{\textsc{QT}}(4n+2;y) = y A_{\textsc{HT}}(2n+1;y) A(n+1;y) A(n;y).
  \end{displaymath}
\end{conjecture}

\subsection{Link patterns}

Any FPL $f$ of size $N$ has a \emph{link pattern}, which is a
partition of the set of integers $1$ to $2N$ into pairs, defined
as follows: first label the endpoints of the open loops $1$ to $2N$ in
clockwise or counterclockwise order (for definiteness, we use
counterclockwise order, starting with the top left endpoint); then the
link pattern will include pair $\{i,j\}$ if and only if the FPL
contains a loop whose two endpoints are labeled $i$ and $j$. Because
the loops are noncrossing, the link pattern satisfies the
\emph{noncrossing} condition: if a link pattern contains two pairs
$\{i,j\}$ and $\{k,\ell\}$, then one cannot have $i<k<j<\ell$. The
possible link patterns for FPLs of size $N$ are counted by the Catalan
numbers $C_N = \frac{1}{N+1}\binom{2N}{N}$, and an easy encoding of
link patterns by Dyck words (or well-formed parenthese words) is as
follows: if $\{i,j\}$ is one of the pairs of the pattern with $i<j$,
the $i$-th letter of the Dyck word is an $a$ (which stands for an
opening parenthese) while the $j$-th letter is a $b$ (closing parenthese).

If an FPL is invariant under a half-turn rotation, then clearly its
link pattern $\pi$ has a symmetry property: if $\{i,j\} \in \pi$, then
$\{i+N,j+N\}\in \pi$ (taking integers modulo $2N$). If $N$ is odd,
the partition is into an odd number of pairs, and exactly one
pair will be of the form $\{i,i+N\}$; if $N$ is even, no pair of the
form $\{i,i+N\}$ will be present. This symmetry lets one encode a
half-turn-invariant link pattern with a word $w$ of length $N$ instead of
$2N$, as follows: for $1\leq i\leq N$,
\begin{itemize}
  \item if $i$ is matched with $j$ with $i<j<j+N$, then the $i$-th
  letter is an $a$;
  \item if $i$ is matched with $i+N$ (odd $N$), then the $i$-th letter
  is a $c$;
  \item otherwise, $i$ is matched with $j$ where $j<i$ or $j>i+N$, and
  the $i$-th letter is a $b$.
\end{itemize}

It is easy to check that, for even $N$, the word $w$ has $N/2$
occurrences of $a$ and $b$, and is thus a \emph{bilateral Dyck word},
while for odd $N$, it has exactly one occurrence of $c$ and $(N-1)/2$
occurrences of each of $a$ and $b$, and is of the form $w=ucv$. In
this case, $vu$ has to be a Dyck word. Overall, the total number of
possible link patterns\footnote{It is surprisingly nontrivial to prove
that each possible word appears as the link pattern of at least one
HTFPL.} is counted by the unified formula $\frac{N!}{\lfloor
N/2\rfloor! \lceil N/2\rceil!}$.

The $2N$ generators $e_1,\dots, e_{2N}$ of the \emph{cyclic
Temperley-Lieb algebra} act on link patterns of size $N$ FPLs in the
following way: if link pattern $\pi$ contains pairs $\{i,j\}$ and
$\{i+1,k\}$, then $e_i\pi= \pi'$, where $\pi'$ is obtained from $\pi$
by replacing the pairs $\{i,j\}$ and $\{i+1,k\}$ by $\{i,i+1\}$ and
$\{j,k\}$; if $\{i,i+1\}\in\pi$, then $\pi'=\pi$. One easily checks
that the $e_i$ operators satisfy the Templerley-Lieb commutation
relations
\begin{displaymath}
\begin{array}{rclr}
  e_{i}e_{j} & = & e_{j} e_{i} & \hbox{when } |i-j|>1\\
  e_{i} e_{i\pm 1} e_{i} & = & e_{i\pm 1} e_i e_{i\pm 1} & \hbox{for
  any } i\\
  e_{i}^2 & = & e_i & 
\end{array}
\end{displaymath}
(generator indices, just like integers in the link pattern, are used
modulo $2N$). Similarly, the $N$ ``symmetrized'' operators $e'_i = e_i
e_{i+N}$ (for $N\geq 2$) act on the link patterns of HTFPLs of size
$N$, and these $N$ symmetrized operators also satisfy the commutation
relations for the $N$-generator cyclic Temperley-Lieb algebra.

In both the nonsymmetric and half-turn-symmetric cases, one can define
a Markov chain on link patterns where, at each time step, one of the
appropriate generators is chosen uniformly at random and applied to
the current state. In each case, the Markov chain is easily checked to
be irreducible and aperiodic, hence it has a unique stationary
distribution. Recent interest in FPLs and their link patterns is
largely due to the following conjectures:

\begin{conjecture}[Razumov, Stroganov~\cite{RazStr04}]\label{conj:RS}
The stationary distribution for link patterns of size $N$ is
\begin{displaymath}
  \mu(\pi) = \frac{A(N;\pi)}{A(N)}.
\end{displaymath}
\end{conjecture}

\begin{conjecture}[de Gier~\cite{deG05}\label{conj:deG}]
  The stationary distribution for half-turn-invariant link patterns of
  size $N$ is 
  \begin{displaymath}
    \mu_{\textsc{HT}}(\pi) = \frac{A_{\textsc{HT}}(N;\pi)}{A_{\textsc{HT}}(N)}.
  \end{displaymath}
\end{conjecture}

By their definitions, the stationary distributions $\mu$ and
$\mu_{\textsc{HT}}$ are invariant under the ``rotation'' mapping (in
the noncrossing partition view) $i\mapsto i+1 \mod
2N$. Wieland~\cite{Wie00} bijectively proved that the distribution of
link patterns of FPLs also has this property; his bijection maps
HTFPLs to HTFPLs and (even-sized) QTFPLs to QTFPLs, so the same is
true of the distributions of their link patterns. It is easy to check
that the same bijection maps qQTFPLs to qQTFPLs (with the special
provision that it might change the edges around the center square from
``two horizontal edges'' to ``two vertical edges'', so the edges
around this center square might have to be inverted).

\section{Link patterns of QTFPLs}
\label{sec:QTFPL_link}
Let $N=4n$. Any QTFPL $f$ of size $N$ is also a HTFPL, so its link
pattern can be described by a bilateral Dyck word $w$ of length
$N$. But, because $f$ is invariant under a quarter-turn rotation, $ww$
must be invariant under conjugation with its left factor $w'$ of
length $N/2$. This means we must have $w=w'.w'$, and thus $w'$ is also
a bilateral Dyck word.

Thus, the link patterns of QTFPLs of size $4n$ can be described by the
same words that we use to describe link patterns of HTFPLs of size
$2n$. We use $\mathcal{A}_{\textsc{QT}}(N;w)$ (where $N$ is divisible
by 4, and $w$ is a bilateral Dyck word of length $N/2$) to denote the
set of all QTFPLs of size $N$ with link pattern $w$ (or link pattern
$w.w$ when viewed as HTFPLs), and $A_{\textsc{QT}}(N;w)$ to denote its
cardinality.

We conjecture the following:

\begin{conjecture}\label{conj:QTFPL}
  For any $n\geq 0$ and bilateral Dyck word $w$ of length $2n$,
  \begin{equation}
    A_{\textsc{QT}}(4n;w) = A_{\textsc{HT}}(2n;w). A(n)^2 \label{equa:QTFPL_pattern}.
  \end{equation}
\end{conjecture}

In other words, the link patterns of even-sized QTFPLs are distributed
\emph{exactly} as those of HTFPLs with half their size. 


Conjecture~\ref{conj:QTFPL} has been checked by exhaustive enumeration
up to $k=5$ (there are 114640611228 QTFPLs of size 20; the next term
in the sequence is 10995014015567296, which makes exhaustive
generation unreasonable).

When $f\in\mathcal{A}_{\textsc{QT}}(4n+2)$ is a qQTFPL, it is also a
HTFPL and its link pattern as such is described by a bilateral Dyck
word of length $4N+2$. But, again, the link pattern is of a special
form: because of the rotational symmetry, the paths entering the
center square of the grid by its four corners are rotational images of
each other, and cannot form closed loops. Thus, these paths exit the
grid at 4 endpoints, which form a single orbit under the quarter-turn
rotation. Furthermore, the HTFPL link pattern is necessarily of the
form $uavubv$ or $ubvuav$, where $vu$ is a Dyck word of length $2n$
(this implies that the factorization is unique). If we retain only the
first $2n+1$ letters of this word, and replace the distinguished $a$
or $b$ letter with a $c$, what we obtain is exactly the link pattern
of a HTFPL of sized $2n+1$; this is what we hereafter call the link
pattern of $f$. Note that if, in the definition of qQTFPLs, we
required that the center square have vertical edges instead of
horizontal edges, this would only change link patterns as HTFPLs
(patterns of the form $uavubv$ would become $ubvuav$ would become
$uavubv$, and vice versa) but not as qQTFPLs.

 As an example, the qQTFPL shown in Figure~\ref{fig:examples}(a) has
link pattern $babca$ as a qQTFPL, $babaababba$ as a HTFPL, and
$aabaababbababaababbb$ as a full FPL link pattern.

With this convention, we have a conjecture for the link patterns of
qQTFPLs of size $4n+2$, relating them to those of HTFPLs of size
$2n+1$:

\begin{conjecture}\label{conj:qQTFPL}
  For any $n\geq 0$ and any half-turn-invariant link pattern $w$ of
  length $2n+1$,
  \begin{equation}
    A_{\textsc{QT}}(4n+2;w) = A_{\textsc{HT}}(2n+1;w) A(n+1) A(n).\label{equa:qQTFPL_pattern}
  \end{equation}
\end{conjecture}

Of course, summation over all link patterns gives
Conjecture~\ref{conj:qQTFPL_comptage}, and this can be interpreted as
saying that link patterns of qQTFPLs of size $4n+2$ are distributed
exactly as those of HTFPLs of size $2n+1$.

Conjectures~\ref{conj:qQTFPL_comptage}, \ref{conj:qQTFPL_firstOne} and
\ref{conj:qQTFPL} have been checked by exhaustive enumeration
up to $n=4$; the total number of qQTFPLs of size 18 is $39204\cdot
429\cdot 42 = 706377672$. The next term in the conjectured sequence,
$7422987\cdot 7436 \cdot 429 = 23 679 655 141 428$, is out of reach of
exhaustive enumeration programs.

A note on terminology: in the rest of this paper, whenever we mention
the link pattern of a QTFPL or qQTFPL, it should be understood to mean
the word with length half the size of the FPL; if we need to reference
the link pattern as an FPL (which is a Dyck word with length double
the size of the FPL), we will write \emph{full link pattern}.

\section{A special case: the rarest link pattern}
\label{sec:special_cases}

In this section, we prove special cases of
Conjectures~\ref{conj:QTFPL} and \ref{conj:qQTFPL} when the considered
link pattern is a very specific one. When the wanted link pattern is
of the form $b^{n}a^{n}$ (for QTFPLs) or $b^{n}ca^{n}$ (for qQTFPLs), the
HTFPLs whose enumeration appear in the conjectures have full link
pattern $a^{2n}b^{2n}$ or $a^{2n+1}b^{2n+1}$, respectively. In each
case, there is only one HTFPL with such a full link pattern; in fact,
there is only one FPL with such a link pattern (this has been noticed
by many authors; one easy way to properly prove it is with the fixed
edge technique of Caselli and Krattenthaler which we use below). Thus,
to prove the corresponding special cases of
Conjectures~\ref{conj:QTFPL} and \ref{conj:qQTFPL}, we only need to
prove that the corresponding QTFPLs and qQTFPLs are counted by
$A(n)^2$ and $A(n)A(n+1)$, respectively.  Both proofs are through a
bijection with a specific class of plane partitions.

For our purposes, a \emph{plane partition of size $k$} is a tiling of
the regular hexagon $H_k$ of side $k$ with rhombi of unit side. When the
hexagon is tiled with equilateral triangles of unit side, the dual
graph is a region $R_k$ of the honeycomb lattice, and rhombus tilings are in
natural bijection with perfect matchings of $R_k$.

A plane partition is said to be \emph{cyclically symmetric} if the
tiling is invariant under a rotation of 120 degrees, and
\emph{self-complementary} if it is invariant under a central symmetry
(the terminology is somewhat confusing when plane partitions are
viewed as tilings, but it is standard). Thus, \emph{cyclically
symmetric, self-complementary} plane partitions (CSSCPPs for short)
are those that are invariant under a rotation of 60 degrees. It is easy
to see that CSSCPPs only exist for even sizes, and it is known that
the number of CSSCPP of size $2n$ is equal to $A(n)^2$.

\begin{figure}[htbp]
  \begin{center}
    \begin{pspicture}(4,4)
\rput[B](2,0){(a)}
\psset{unit=.3cm}
\psline{*-*}(6,6)(7,6)
\psline{*-*}(6,7)(7,7)
\psline(1,3)(2,3)(2,4)(4,4)(4,3)(3,3)(3,2)(2,2)(2,1)
\psline(4,1)(4,2)(6,2)(6,1)
\psline(8,1)(8,2)(7,2)(7,3)(8,3)(8,4)(7,4)(7,5)(8,5)(8,6)(6,6)
\psline(6,6)(6,5)(5,5)(5,6)(4,6)(4,5)(3,5)(3,6)(2,6)(2,5)(1,5)
\psline(5,3)(5,4)(6,4)(6,3)(5,3)
\psline(10,1)(10,2)(9,2)(9,4)(10,4)(10,3)(11,3)(11,2)(12,2)
\psline(12,4)(11,4)(11,6)(12,6)
\psline(10,5)(10,6)(9,6)(9,5)(10,5)
\psline(12,8)(11,8)(11,7)(10,7)(10,8)(9,8)(9,7)(8,7)(8,8)(7,8)(7,7)
\psline(7,7)(5,7)(5,8)(6,8)(6,9)(5,9)(5,10)(6,10)(6,11)(5,11)(5,12)
\psline(7,9)(8,9)(8,10)(7,10)(7,9)
\psline(12,10)(11,10)(11,9)(9,9)(9,10)(10,10)(10,11)(11,11)(11,12)
\psline(9,12)(9,11)(7,11)(7,12)
\psline(3,12)(3,11)(4,11)(4,9)(3,9)(3,10)(2,10)(2,11)(1,11)
\psline(1,7)(2,7)(2,9)(1,9)
\psline(3,7)(4,7)(4,8)(3,8)(3,7)
\end{pspicture}\hspace{1cm}
    \begin{pspicture}(4,4)
\rput[B](1.7,0){(b)}
\psset{xunit=0.075cm,yunit=0.1249cm,linewidth=0.5pt,origin={0,-2.5}}
\psline[linewidth=1pt](0,7)(21,0)(42,7)(42,21)(21,28)(0,21)(0,7)
\psline[linewidth=1pt](21,12)(24,13)(24,15)(21,16)(18,15)(18,13)(21,12)
\psline(0,9)(15,4)(18,5)(21,4)(33,8)(36,7)(42,9)
\psline(0,11)(3,10)(6,11)(12,9)(15,10)(18,9)(21,10)(27,8)(36,11)(39,10)(42,11)
\psline(0,13)(3,14)(12,11)(18,13)(24,11)(27,12)(30,11)(36,13)(39,12)(42,13)
\psline(0,15)(3,16)(9,14)(12,15)(15,14)(21,16)(24,15)(27,16)(30,15)(33,16)(39,14)(42,15)
\psline(0,17)(3,18)(6,17)(12,19)(15,18)(18,19)(24,17)(30,19)(39,16)(42,17)
\psline(0,19)(3,20)(6,19)(15,22)(21,20)(24,21)(27,20)(30,21)(36,19)(39,20)(42,19)
\psline(18,1)(33,6)(36,5)
\psline(15,2)(18,3)(24,1)
\psline(24,3)(27,2)\psline(27,4)(30,3)\psline(30,5)(33,4)
\psline(3,8)(6,9)(18,5)(24,7)\psline(6,7)(9,8)\psline(9,6)(15,8)(24,5)\psline(12,5)(21,8)(27,6)
\psline(0,11)(3,12)(6,11)
\psline(15,10)(18,11)(21,10)
\psline(24,9)(27,10)(30,9)
\psline(33,8)(36,9)(39,8)
\psline(9,12)(12,13)(15,12)
\psline(21,12)(27,14)(33,12)\psline(24,13)(27,12)(33,14)(36,13)
\psline(6,15)(12,17)(18,15)\psline(9,16)(12,15)(18,17)(21,16)
\psline(27,16)(30,17)(33,16)
\psline(12,19)(15,20)(18,19)
\psline(21,18)(24,19)(27,18)
\psline(36,17)(39,18)(42,17)
\psline(3,20)(6,21)(9,20)
\psline(15,22)(21,24)(30,21)(33,22)\psline(18,21)(27,24)(39,20)\psline(18,23)(24,21)(30,23)
\psline(9,24)(12,23)\psline(12,25)(15,24)\psline(15,26)(18,25)\psline(18,27)(21,26)(24,27)\psline(6,23)(9,22)(21,26)(24,25)(27,26)
\psline(3,6)(3,10)\psline(3,12)(3,22)
\psline(6,5)(6,7)\psline(6,9)(6,19)\psline(6,21)(6,23)
\psline(9,4)(9,6)\psline(9,8)(9,14)\psline(9,16)(9,22)
\psline(12,3)(12,5)\psline(12,7)(12,11)\psline(12,13)(12,15)\psline(12,17)(12,23)
\psline(15,2)(15,4)\psline(15,8)(15,14)\psline(15,16)(15,18)\psline(15,20)(15,24)
\psline(18,3)(18,5)\psline(18,7)(18,9)\psline(18,11)(18,15)\psline(18,17)(18,21)\psline(18,23)(18,25)
\psline(21,2)(21,4)\psline(21,8)(21,12)\psline(21,16)(21,20)\psline(21,24)(21,26)
\psline(24,3)(24,5)\psline(24,7)(24,11)\psline(24,13)(24,17)\psline(24,19)(24,21)\psline(24,23)(24,25)
\psline(27,4)(27,8)\psline(27,10)(27,12)\psline(27,14)(27,20)\psline(27,24)(27,26)
\psline(30,5)(30,11)\psline(30,13)(30,15)\psline(30,17)(30,21)\psline(30,23)(30,25)
\psline(33,6)(33,12)\psline(33,14)(33,20)\psline(33,22)(33,24)
\psline(36,5)(36,7)\psline(36,9)(36,19)\psline(36,21)(36,23)
\psline(39,6)(39,16)\psline(39,18)(39,22)
\psline(33,20)(36,21)
\end{pspicture}
  \end{center}
  \caption{(a) Example qQTFPL of size 10 and (b) Example qCSSCPP of size $7$}
  \label{fig:examples}
\end{figure}
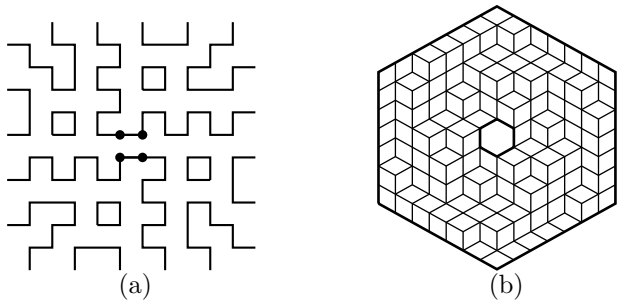

We define a \emph{quasi-cyclically symmetric, self-complementary plane
  partition} (qCSSCPP) of size $2n+1$ as a rhombus tiling, invariant
  under rotation of 60 degrees, of the regular hexagon of size $2n+1$
  with the central unit side hexagon removed. Such tilings do not
  appear to have been previously studied in the literature.

We will prove the following:

\begin{theorem}\label{th:QTpattern_PP}
For any $n\geq 1$, there is a bijection between
$\mathcal{A}_{\textsc{QT}}(4n;b^na^n)$ and the set of CSSCPPs of size
$2n$, and a bijection between
$\mathcal{A}_{\textsc{QT}}(4n+2;b^nca^n)$ and the set of qCSSCPPs of
size $2n+1$.
\end{theorem}

The known enumeration of CSSCPPs then concludes the proof of
(\ref{equa:QTFPL_pattern}) for pattern $b^na^n$; to prove
(\ref{equa:qQTFPL_pattern}) for pattern $b^nca^n$, we will need
our last theorem:

\begin{theorem}\label{th:enum_qCSSCPP}
  The number of qCSSCPPs of size $2n+1$ is $A(n)A(n+1)$.
\end{theorem}





\begin{proof}[of Theorem~\ref{th:QTpattern_PP}]
  We rely on the technique of ``fixed edges'' as used by Caselli and
  Krattenthaler in~\cite{CasKra04} (see
  also~\cite{CasKraLasNad05}). The technique uses the fact that, for a
  given link pattern, there may be a large set of edges which appear
  in all FPLs with this particular link pattern. In some cases, this
  makes it possible to find a bijection between the target set of
  FPLs and the perfect matchings of some particular planar graph,
  typically a region of the hexagonal lattice.

  QTFPLs with link pattern $b^na^n$ have the full link pattern
  $a^n(a^nb^n)^3b^n$, which consists of 4 sets of $n$ nested arches
  each. This means that, on each of the grid sides, all $n$ outgoing
  links are forbidden from connecting to each other; thus, by
  Lemma~3.1 of \cite{CasKraLasNad05}, the following edges are fixed in
  every FPL $f\in \mathcal{A}(4n; a^n(a^nb^n)^3 b^n)$:

  \begin{itemize}
    \item each horizontal edge whose \emph{left} endpoint has odd sum
    of coordinates, inside the triangle whose vertices have
    coordinates $(0,0)$, $(4n-2,0)$ and $(2n-1,2n-1)$ (triangle $ABC$
    on Figure~\ref{fig:fixed_edges}(a));
    \item their orbits under the action of the 90 degree rotation
    centered at $(2n-1/2,2n-1/2)$.
  \end{itemize}

  \begin{figure}[htbp]
    \begin{center}
      \epsfig{file=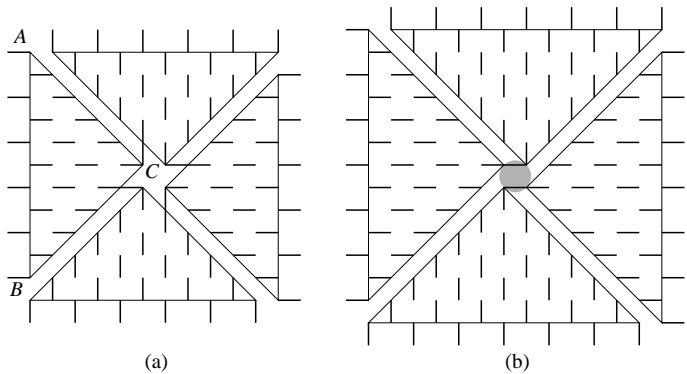,scale=0.6}
      \caption{Fixed edges for (a)
      $\mathcal{A}_{\textsc{QT}}(12,bbbaaa)$ (b)
      $\mathcal{A}_{\textsc{QT}}(14,bbbcaaa)$}
      \label{fig:fixed_edges}
    \end{center}
  \end{figure}

  Similarly, qQTFPLs with link pattern $b^nca^n$ have as their full link
  pattern $a^{2n+1}b^{n}a^{n+1}b^{n+1}a^{n}b^{2n+1}$ (four sets of
  nested arches, with alternatingly $n+1$, $n$, $n+1$ and $n$ arches
  each). Again, the same fixed edges appear, as shown in
  Figure~\ref{fig:fixed_edges}(b).

  \begin{figure}[htbp]
    \begin{center}
      \epsfig{file=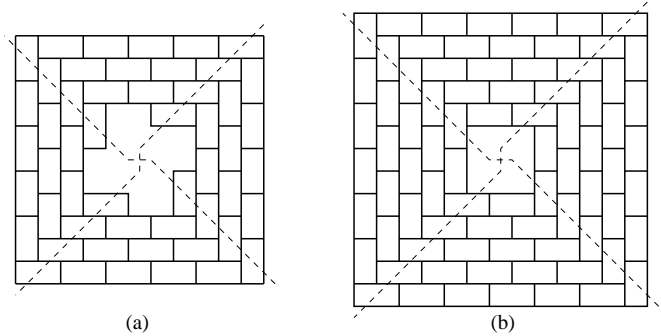,scale=0.6}
    \end{center}
\caption{Graphs of non-fixed edges for (a)
  $\mathcal{A}_{\textsc{QT}}(12,bbbaaa)$ (b)
  $\mathcal{A}_{\textsc{QT}}(14,bbbcaaa)$}
\label{fig:nonfixed_edges}
  \end{figure}

  In a QTFPL of size $4n$, the four ``center'' edges joining vertices
  $C$ and its rotational images must either all be included, or all
  excluded; with these vertices already having one incident fixed edge
  each, they must be excluded. This in turn forces the presence of
  four more fixed edges, as shown in
  Figure~\ref{fig:fixed_edges}(a). (The symmetry conditions also force
  additional edges in the corners, but it is slightly more convenient
  to not mention them now.)

  To prove that we indeed have a bijection, we need to check that all
  QTFPLs (respectively, qQTFPLs) sharing the above-mentioned edges
  have link pattern $b^na^n$) (respectively, $b^nca^n$). We do this in
  detail for the QTFPL case; the proof for qQTFPL is similar.

  Consider the $2n$ paths starting from endpoints along segment
  $AB$. The horizontal fixed edges inside triangle $ABC$ prevent them
  from connecting with each other, so that each of them will exit
  triangle $ABC$ either to the top (through segment $AC$, including
  $C$ but excluding $A$) or to the bottom (through segment $CB$,
  including $B$ but excluding $C$. Any path exiting through segment
  $AC$ will be connected to one from the top border, while any path
  exiting through segment $CB$ will be connected to one from the
  bottom border. Thus, the link pattern will be of the form
  $b^{k}a^{2n-k}$, where $k$ is the number of paths exiting along
  segment $AC$. But quarter-turn symmetry implies that the number of
  paths exiting triangle $ABC$ through segment $AC$ is equal to the
  number of paths entering triangle $ABC$ from the bottom triangle;
  thus, $k=2n-k$, and the link pattern is indeed $b^{n}a^{n}$.

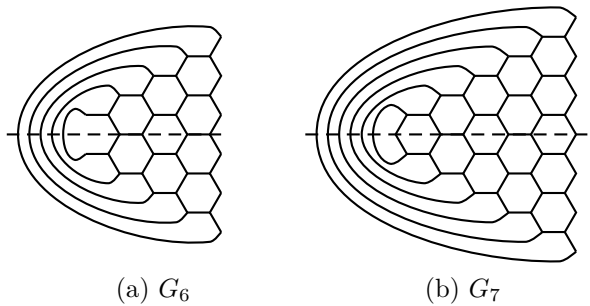
\begin{figure}[htbp]
\begin{center}
  \begin{pspicture}(4,4)
    \psset{xunit=0.15cm,yunit=0.2598cm}
    \psline(4,7)(6,7)(7,6)(9,6)(10,5)(12,5)(13,4)(15,4)(16,3)
    \psline(4,9)(6,9)(7,8)(9,8)(10,7)(12,7)(13,6)(15,6)(16,5)
    \psline(6,7)(7,8)
    \psline(9,6)(10,7)
    \psline(12,5)(13,6)
    \psline(15,4)(16,5)
    \psline(6,9)(7,10)
    \psline(9,8)(10,9)
    \psline(12,7)(13,8)
    \psline(15,6)(16,7)
    \psline(7,10)(9,10)(10,9)(12,9)(13,8)(15,8)(16,7)
    \psline(9,10)(10,11)
    \psline(12,9)(13,10)
    \psline(15,8)(16,9)
    \psline(10,11)(12,11)(13,10)(15,10)(16,9)
    \psline(12,11)(13,12)(15,12)(16,11)(15,10)
    \psline(15,12)(16,13)
    \pscurve(4,9)(3,9.3)(2,8)(3,6.7)(4,7)
    \pscurve(7,10)(6,10.5)(1,8)(6,5.5)(7,6)
    \pscurve(10,11)(9,11.5)(0,8)(9,4.5)(10,5)
    \pscurve(13,12)(12,12.5)(-1,8)(12,3.5)(13,4)
    \pscurve(16,13)(15,13.5)(-2,8)(15,2.5)(16,3)
    \rput(10,0){(a) $G_{6}$}
    \psline[linestyle=dashed](-3,8)(16,8)
  \end{pspicture}
  \begin{pspicture}(4,4)
    \psset{xunit=0.15cm,yunit=0.2598cm}
    \psline(20,2)(19,3)(17,3)(16,4)(14,4)(13,5)(11,5)(10,6)(8,6)(7,7)
    \psline(7,7)(5,7)(4,8)(5,9)(7,9)(8,10)(10,10)(11,11)(13,11)(14,12)
    \psline(14,12)(16,12)(17,13)(19,13)(20,14)
    \psline(7,7)(8,8)(7,9)
    \psline(8,8)(10,8)(11,7)(13,7)(14,6)(16,6)(17,5)(19,5)(20,4)(19,3)
    \psline(10,6)(11,7)
    \psline(13,5)(14,6)
    \psline(16,4)(17,5)
    \psline(10,8)(11,9)(13,9)(14,8)(16,8)(17,7)(19,7)(20,6)(19,5)
    \psline(13,7)(14,8)
    \psline(16,6)(17,7)
    \psline(11,9)(10,10)
    \psline(13,9)(14,10)(16,10)(17,9)(19,9)(20,8)(19,7)
    \psline(14,10)(13,11)
    \psline(16,8)(17,9)
    \psline(19,9)(20,10)(19,11)(20,12)(19,13)
    \psline(16,10)(17,11)(16,12)
    \psline(17,11)(19,11)
    \pscurve(5,9)(3.5,9.5)(2,8)(3.5,6.5)(5,7)
    \pscurve(8,10)(6.5,10.5)(1,8)(6.5,5.5)(8,6)
    \pscurve(11,11)(9.5,11.5)(0,8)(9.5,4.5)(11,5)
    \pscurve(14,12)(12.5,12.5)(-1,8)(12.5,3.5)(14,4)
    \pscurve(17,13)(15.5,13.5)(-2,8)(15.5,2.5)(17,3)
    \pscurve(20,14)(18.5,14.5)(-3,8)(18.5,1.5)(20,2)
    \rput(10,0){(b) $G_{7}$}
    \psline[linestyle=dashed](-4,8)(21,8)
  \end{pspicture}
\end{center}
\caption{Quotients under rotation of the nonfixed edge graphs in
Figure~\ref{fig:nonfixed_edges}}
\label{fig:hexgraphs}
\end{figure}

  Note that in both the QTFPL and qQTFPL cases, each vertex in the
  grid is incident to either 1 or 2 fixed edges. Thus, if we delete
  from the grid the fixed edges and the ``forbidden'' edges (those
  non-fixed edges that are incident to at least one vertex with two
  incident fixed edges), we get a graph whose rotationally invariant
  perfect matchings are in bijection with the considered symmetric
  FPLs. These two graphs, shown on Figure~\ref{fig:nonfixed_edges},
  naturally have a rotational symmetry of order $4$, so we need only
  consider the perfect matchings of their orbit graphs under this
  rotational symmetry, which are shown on Figure~\ref{fig:hexgraphs}
  (the dashed lines in Figure~\ref{fig:nonfixed_edges} show where to
  ``cut'' to obtain the quotients). For QTFPLs, this quotient graph is
  exactly the orbit graph, under rotational symmetry of order $6$, of
  the honeycomb graph $R_{2n}$; for qQTFPLs, it is the orbit graph,
  under rotational symmetry of order $6$, of the ``holed'' honeycomb
  graph $R'_{2n+1}$. Putting all pieces together, we have the required
  bijections between FPLs and plane partitions.
\end{proof}

\begin{proof}[of Theorem~\ref{th:enum_qCSSCPP}]
  We now turn to the enumeration of qCSSCPPs of size $2n+1$, for which
we know that they are in bijection with the perfect matchings of the
quotiented honeycomb lattice region $G_{2n+1}$.

Notice that $G_{2n+1}$, as shown in Figure~\ref{fig:hexgraphs}, has a
reflective symmetry, with $2n$ vertices on the symmetry axis. Thus, we
can use Ciucu's Matching Factorization Theorem~\cite{Ciu97} (or,
rather, the slight generalization proved in Section 7 of \cite{Ciu97},
and used, in a very similar context to ours, in \cite{Ciu99}), and we
get that the number of perfect matchings of $G_{2n+1}$ is $2^n
M^{*}(G'_{2n+1})$, where $G'_{2n+1}$ is $G_{2n+1}$ with all edges incident
to the symmetry axis, and lying below it, removed, and edges lying on
the symmetry axis weighted $1/2$; and $M^{*}(G)$ denotes the weighted
enumeration of perfect matchings of $G$, that is, the sum over perfect
matchings of the product of weights of selected edges.

\begin{figure}[htbp]
  \begin{center}
    \begin{pspicture}(4,4)
      \put(1.5,0){(a)}
      \psset{xunit=0.2cm,yunit=0.3464cm}
      \psline(0,2)(1,3)(3,3)(4,2)(6,2)(7,3)(9,3)(10,2)(12,2)(13,3)(15,3)(16,2)
      \psline(3,3)(4,4)
      \psline(7,3)(6,4)
      \psline(9,3)(10,4)
      \psline(13,3)(12,4)
      \psline(1,3)(0,4)(1,5)(3,5)(4,4)(6,4)(7,5)(9,5)(10,4)(12,4)(13,5)(15,5)(16,4)(15,3)
      \psline(3,5)(4,6)(6,6)(7,5)
      \psline(9,5)(10,6)(12,6)(13,5)
      \psline(6,6)(7,7)
      \psline(10,6)(9,7)
      \psline(12,6)(13,7)
      \psline(4,6)(3,7)(4,8)(6,8)(7,7)(9,7)(10,8)(12,8)(13,7)(15,7)(16,6)(15,5)
      \psline(6,8)(7,9)(9,9)(10,8)
      \psline(9,9)(10,10)
      \psline(12,8)(13,9)
      \psline(7,9)(6,10)(7,11)(9,11)(10,10)(12,10)(13,9)(15,9)(16,8)(15,7)
      \uput{3pt}[d](5,2){$\scriptstyle 1/2$}
      \uput{3pt}[d](11,2){$\scriptstyle 1/2$}
      \psset{linestyle=dotted,linewidth=0.5pt}
      \psline(-1,1)(17,7)
      \psline(5,1)(17,5)
      \psline(11,1)(17,3)
      \psline(-1,3)(17,9)
      \psline(-1,5)(14,10)
      \psline(2,8)(11,11)
      \psline(5,11)(8,12)
      \psline(-1,1)(-1,5)
      \psline(2,2)(2,8)
      \psline(5,1)(5,11)
      \psline(8,2)(8,12)
      \psline(11,1)(11,11)
      \psline(14,2)(14,10)
      \psline(17,1)(17,9)
      \psline(-1,3)(5,1)
      \psline(-1,5)(11,1)
      \psline(2,6)(17,1)
      \psline(2,8)(17,3)
      \psline(5,9)(17,5)
      \psline(5,11)(17,7)
      \psline(8,12)(17,9)
    \end{pspicture}
    \begin{pspicture}(4,4)
      \put(1.5,0){(b)}
      \psset{xunit=0.2cm,yunit=0.3464cm}
      \psdiamond[linestyle=none,fillstyle=solid,fillcolor=lightgray](5,2)(3,1)
      \psdiamond[linestyle=none,fillstyle=solid,fillcolor=lightgray](11,2)(3,1)
      \psline(-1,1)(17,7)
      \psline(5,1)(17,5)
      \psline(11,1)(17,3)
      \psline(-1,3)(17,9)
      \psline(-1,5)(14,10)
      \psline(2,8)(11,11)
      \psline(5,11)(8,12)
      \psline(-1,1)(-1,5)
      \psline(2,2)(2,8)
      \psline(5,1)(5,11)
      \psline(8,2)(8,12)
      \psline(11,1)(11,11)
      \psline(14,2)(14,10)
      \psline(17,1)(17,9)
      \psline(-1,3)(5,1)
      \psline(-1,5)(11,1)
      \psline(2,6)(17,1)
      \psline(2,8)(17,3)
      \psline(5,9)(17,5)
      \psline(5,11)(17,7)
      \psline(8,12)(17,9)
      \psset{linestyle=dotted,linewidth=0.5pt}
      \psline{*-*}(0.5,5.5)(0.5,1.5)
      \psline{*-}(3.5,8.5)(3.5,2.5)
      \psline{*-*}(6.5,11.5)(6.5,1.5)
      \psline(9.5,10.5)(9.5,2.5)
      \psline{-*}(12.5,9.5)(12.5,1.5)
      \psline(0.5,3.5)(6.5,1.5)
      \psline(0.5,5.5)(12.5,1.5)
      \psline(3.5,6.5)(12.5,3.5)
      \psline(3.5,8.5)(12.5,5.5)
      \psline(6.5,9.5)(12.5,7.5)
      \psline(6.5,11.5)(12.5,9.5)
    \end{pspicture}
    \begin{pspicture}(4,5)
      \put(1.5,0){(c)}
      \psset{origin={0,-0.5},unit=0.6cm}
      \psline{*-*}(0,0)(0,2)
      \psline{-*}(1,1)(1,4)
      \psline{*-*}(2,1)(2,6)
      \psline(3,2)(3,6)
      \psline{*-}(4,2)(4,6)
      \psline(0,1)(2,1)
      \psline(0,2)(4,2)
      \psline(1,3)(4,3)
      \psline(1,4)(4,4)
      \psline(2,5)(4,5)
      \psline(2,6)(4,6)
      \psline[linewidth=1.5pt](1,1)(2,1)
      \psline[linewidth=1.5pt](3,2)(4,2)
      \uput{3pt}[d](1.5,1.8){$\scriptstyle 1/2$}
      \uput{3pt}[d](3.5,2.8){$\scriptstyle 1/2$}
      \uput[ul](0,2.5){$A_0$}
      \uput[ul](1,4.5){$A_1$}
      \uput[ul](2,6.5){$A_2$}
      \uput[dr](0,1){$B_0$}
      \uput[dr](2,2){$B_1$}
      \uput[dr](4,3){$B_2$}
    \end{pspicture}
  \end{center}
  \caption{(a) Honeycomb region $G'_{7}$, (b) Triangular lattice
  $R'_{7}$ and (c) corresponding square lattice points}
  \label{fig:RegionTriangqCSSCPP}
\end{figure}
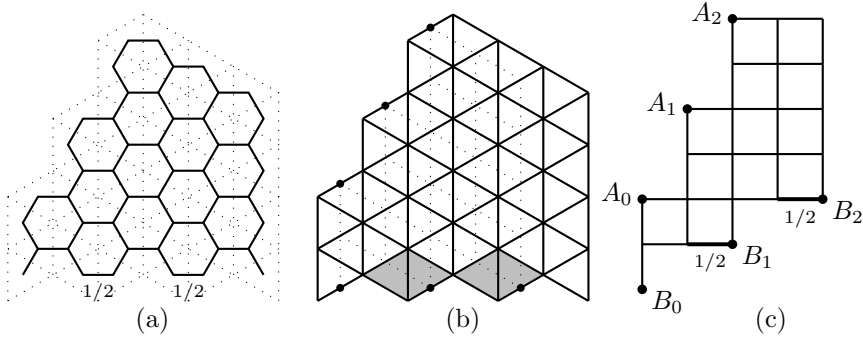

$G'_{2n+1}$, when redrawn as a region of the honeycomb lattice, is the
dual of region $R'_{2n+1}$ of the triangular lattice (shown on
Figure~\ref{fig:RegionTriangqCSSCPP}(b), with the weight $1/2$ rhombi
greyed), on which we need to count rhombi tilings. $n+1$ rhombi on the
right border of $R'_{2n+1}$ are fixed (will appear in all
tilings). Using a classical correspondence between rhombi tilings and
weighted configurations of nonintersecting lattice paths, we are left
with counting the number of nonintersecting (square) lattice path
configurations, where the paths, using East and South unit steps,
collectively join vertices $A_i, 0\leq i\leq n-1$, to vertices $B_j,
0\leq j\leq n-1$, with respective coordinates $(i,2i+2)$ and $(2j,j)$;
the horizontal edges with right endpoints $B_j$ carry a weight $1/2$,
so that the weighted enumeration of paths joining vertices $A_i$ and
$B_j$ is
\begin{eqnarray*}
  w(A_i,B_j) & = & \frac{1}{2} \binom{i+j+1}{2j-i} + \binom{i+j+1}{2j-i-1}\\
 &  = & \frac{1}{2} \left( \binom{i+j+1}{2j-i} + \binom{i+j+2}{2j-i}
 \right)\\
 & = & \frac{1}{2} (3i+4) \frac{(i+j+1)!}{(2j-i)! (2i-j+2)!}.
\end{eqnarray*}

The Lindstr\"om-Gessel-Viennot theorem~\cite{Lin73,GesVie85} now expresses
$M^{*}(G'_{2n+1})$ as the determinant
\begin{equation}
  M^{*}(G'_{2n+1}) = \det\left( w(A_i,B_j) \right)_{0\leq i,j\leq n-1};
\end{equation}
factoring out $\frac{3i+4}{2}$ in line $i$ of the matrix, we get the
number of perfect matchings of $G_{2n+1}$ as
\begin{equation}\label{bigdet}
  M(G_{2n+1}) = \left(\prod_{i=0}^{n-1} 3i+4\right) \det\left(
  \frac{(i+j+1)!}{(2j-i)! (2i-j+2)!}  \right)_{0\leq i,j\leq n-1}.
\end{equation}

The determinant in (\ref{bigdet}) happens to be the special case $x=2,
y=0$ of \cite[Theorem 40]{Kra99} (the enumeration of CSSCPPs of size
$2n$ by the same method, as in \cite{Ciu99}, corresponds to
$x=1,y=0$), and evaluates to
\begin{displaymath}
  \prod_{0\leq i\leq n-1} \frac{i! (i+1)! (3i+3)! (3i+1)!}{(2i+2)!
  (2i)! (2i+3)! (2i+1)!},
\end{displaymath}
so that the number of qCSSCPP of size $2n+1$ is
\begin{equation}
  p_{2n+1} = \prod_{i=0}^{n-1} \frac{i! (i+1)! (3i+1)! (3i+4)!}{(2i)!
  (2i+1)! (2i+2)! (2i+3)!}.
\end{equation}

To finish the proof that $p_{2n+1}=A(n)A(n+1)$, we need only check
that the ratio of two consecutive enumerations is as predicted (the
case $n=1$ corresponds to checking that there are only 2 qCSSCPPs of
size $3$):
\begin{eqnarray*}
  \frac{p_{2n+1}}{p_{2n-1}} & = & \frac{(n-1)! n! (3n-2)! (3n+1)!}{
  (2n-2)! (2n-1)! (2n)! (2n+1)!}\\
 & = & \frac{ (n-1)! (3n-2)!}{ (2n-2)! (2n-1)!} . \frac{ n! (3n+1)!}{
  (2n)! (2n+1)!} \\
 & = & \frac{A(n)}{A(n-1)} . \frac{A(n+1)}{A(n)}.
\end{eqnarray*}
\end{proof}

\section{Further comments}

The starting point and first motivation for this paper, as the title
suggests, was Conjecture~\ref{conj:QTFPL}, which nicely complements
the previous conjectures of Razumov-Stroganov and de Gier. This
suggests that a combinatorial proof of one of the conjectures might be
adapted to yield proofs of all of them, possibly by explicitly
devising operators on FPLs that project to the $e_i$ operators on link
patterns, while having suitable bijective properties.

The definitions of qQTFPLs and qCSSCPPs evolved out of an attempt to
devise a general framework for the random generation of symmetric FPLs
and plane partitions~\cite{DucGenFPL}; their enumerative properties
came as a total surprise. It might be possible to prove
Conjecture~\ref{conj:qQTFPL_comptage} by adapting Kuperberg's
methods~\cite{Kup02}, thus bringing one more (quasi-)symmetry class
under the same roof; the author's first attempts in this direction
were unsuccessful.

To the author's knowledge, Theorem~\ref{th:QTpattern_PP} is one of the
first explicit bijections between classes of FPLs (or alternating-sign
matrices) and plane partitions, even though there are many known
(proved or conjectured) enumerative formulae linking the two families
of combinatorial objects.

\bibliographystyle{plain}
\bibliography{bibASM}

\end{document}